\documentclass[11pt,a4paper]{article}
\usepackage{helvet}         
\usepackage{courier}        
\usepackage{makeidx}         
\usepackage{graphicx}        
\usepackage{multicol}        
\usepackage[bottom]{footmisc}
\usepackage{amsmath,amssymb,amsfonts,amsthm}     

\usepackage{amssymb,amsmath, wasysym, graphicx, epsfig, setspace, wrapfig, comment}
\usepackage{latexsym,dsfont,amsfonts,amsmath,amssymb}
\usepackage{bm}
\usepackage{epsfig,psfrag,color}
\usepackage{mathrsfs}
\usepackage{soul}
\usepackage{mathrsfs}   
\usepackage{tikz}
\usepackage[breaklinks=true]{hyperref}

\usepackage{mathtools}

\newcommand{\nrm}[2]{\ensuremath{\|#1\|_{#2}}}

\newcommand{\dist}{\mathrm{dist}}

\newcommand{\N}[0]{\mathbb{N}}

\newcommand{\R}[0]{\mathbb{R}}

\newlength\figureheight
\newlength\figurewidth

\makeatletter
\newcommand{\BIG}{\bBigg@{3}}
\newcommand{\vast}{\bBigg@{4}}
\newcommand{\Vast}{\bBigg@{5}}
\makeatother

\usepackage[normalem]{ulem}

\newtheorem{theorem}{Theorem}
\newtheorem{assumption}{Assumption A$\!\!$}

\newtheorem{Lemma}[theorem]{Lemma}
\newtheorem{Proposition}[theorem]{Proposition}

\newcommand{\sumj}{\sum_{j=1}^\infty}

\newcommand{\bsalpha}{{\boldsymbol{\alpha}}}

\newcommand{\bsx}{{\boldsymbol{x}}}

\newcommand{\bsy}{{\boldsymbol{y}}}
\newcommand{\wbsy}{\widetilde{\boldsymbol{y}}}
\newcommand{\wa}{\widetilde{a}}

\newcommand{\wlam}{\widetilde{\lambda}}

\newcommand{\bsw}{{\boldsymbol{w}}}

\newcommand{\rd}{\,\mathrm{d}}

\newcommand{\calA}{\altmathcal{A}}

\newcommand{\calN}{\altmathcal{N}}

\newcommand{\calV}{\mathcal{V}}

\def\R{\mathbb{R}}

\newcommand{\mask}[1]{{}}

\newcommand{\eps}{\varepsilon}

\definecolor{darkred}{RGB}{139,0,0}
\definecolor{darkgreen}{RGB}{0,100,0}
\definecolor{darkmagenta}{RGB}{170,0,120}
\definecolor{darkpurple}{RGB}{110,0,180}
\definecolor{darkblue}{RGB}{40,0,200}
\definecolor{darkbrown}{rgb}{0.75,0.40,0.15}

\newcommand{\be}{\begin{equation}}
\newcommand{\ee}{\end{equation}}
\newcommand{\bea}{\begin{eqnarray}}
\newcommand{\eea}{\end{eqnarray}}
\newcommand{\beas}{\begin{eqnarray*}}
\newcommand{\eeas}{\end{eqnarray*}}

\DeclareMathAlphabet{\altmathcal}{OMS}{cmsy}{m}{n}

\def\r2p{{\sqrt{2\pi}}}

\newcommand*{\compbed}{\subset\subset}

\newcommand*{\CLip}{\ensuremath{C_\mathrm{Lip}}}

\newcommand*{\wCLip}{\ensuremath{\widetilde{C}_\mathrm{Lip}}}
\newcommand*{\evalueLap}{\ensuremath{\chi_1}}
\newcommand*{\amin}{\ensuremath{a_{\min}}}
\newcommand*{\amax}{\ensuremath{a_{\max}}}

\graphicspath{{.}{./Figures/}{./Numerics/Figures/}}

\title{Bounding the spectral gap for an elliptic eigenvalue problem with uniformly bounded stochastic coefficients}
\date{\today}

\makeatletter
\let\@fnsymbol\@arabic
\makeatother

\author{A. D. Gilbert\footnotemark[1]\and
             I. G. Graham\footnotemark[2] \and
             R. Scheichl\footnotemark[1] $^{,2}$ \and
             I. H. Sloan\footnotemark[3]
             }
                            
\begin{document}
\sloppy
\maketitle

\footnotetext[1]{Institute for Applied Mathematics and Interdisciplinary Center for Scientific Computing,
                           Universit\"{a}t Heidelberg, 69120 Heidelberg, Germany\\
                           \indent\indent\texttt{a.gilbert@uni-heidelberg.de},
                           \texttt{r.scheichl@uni-heidelberg.de}}
\footnotetext[2]{Department of Mathematical Sciences, University of Bath, Bath BA2 7AY UK\\
                          \indent\indent\texttt{i.g.graham@bath.ac.uk}
                          }
\footnotetext[3]{School of Mathematics and Statistics,
                          University of New South Wales, Sydney NSW 2052, Australia\\
                         \indent\indent\texttt{i.sloan@unsw.edu.au}}
                          
\numberwithin{equation}{section}

\abstract{
A key quantity that occurs in the error analysis of several numerical methods for eigenvalue problems
is the distance between the eigenvalue of interest and the next nearest eigenvalue. 
When we are interested in the smallest or fundamental eigenvalue, we call this the \emph{spectral} or \emph{fundamental gap}. 
In a recent manuscript [Gilbert et al., \url{https://arxiv.org/abs/1808.02639}],
the current authors, together with Frances Kuo, studied an elliptic eigenvalue problem
with homogeneous Dirichlet boundary conditions, and
with coefficients that depend on an infinite number of uniformly distributed stochastic parameters.
In this setting, the eigenvalues, and in turn the eigenvalue gap, 
also depend on the stochastic parameters.
Hence, for a robust error analysis 
one needs to be able to bound the gap over all possible realisations of the parameters,
and because the gap depends on infinitely-many random parameters, this is not trivial.
This short note presents, in a simplified setting, an important result that was shown
in the paper above. Namely, that, under certain decay assumptions on the coefficient,  
the spectral gap of such a random elliptic eigenvalue problem can be bounded away from 0, 
uniformly over the entire infinite-dimensional parameter space.
}

\section{Introduction}
Eigenvalue problems are useful for modelling many phenomena from applications in engineering
and physics, e.g, structural mechanics, acoustic scattering, elastic membranes, 
criticality of neutron transport/diffusion and band gap calculations for photonic crystal fibres. 

In this work, we consider the following eigenvalue problem
\begin{align}
\label{eq:evp}
-\nabla\cdot \big(a(\bsx, \bsy) \nabla u(\bsx, \bsy)\big) \,&=\, \lambda(\bsy) u(\bsx, \bsy)\,,
\quad 
&\text{for } \bsx \in D,
\\\nonumber
u(\bsx, \bsy) \,&=\, 0\,, \quad 
&\text{for } \bsx \in \partial D\,,
\end{align}
where the derivatives are taken with respect to $\bsx$, 
and $\bsy = (y_1, y_2, \ldots)$ is a random, infinite-dimensional vector
with independent uniformly distributed components $y_j \sim \mathrm{U}([-\tfrac{1}{2}, \tfrac{1}{2}])$.
We assume that the \emph{physical} domain
$D \subset \R^d$, for $d = 1, 2, 3$, is bounded with Lipschitz boundary, 
and denote the \emph{stochastic/parameter} domain  
by $U \coloneqq [-\tfrac{1}{2}, \tfrac{1}{2}]^\N$.

In many uncertainty quantification (UQ) applications, the 
coefficient $a(\bsx, \bsy)$ is given by a Karhunen--Lo\`eve expansion of a random field.
Taking this as motivation, we assume that the coefficient is an affine map of 
$\bsy$, and satisfies
\[
0 \,<\, a(\bsx, \bsy) \,=\, a_0(\bsx) + \sum_{j = 1}^\infty y_j a_j(\bsx) \, <\, \infty \,,
\quad \text{for all } x \in D,\ \bsy \in U.
\]
Further assumptions on the coefficient will be given
explicitly in Assumption~A\ref{asm:coeff} below.

If we ignore the $\bsy$ dependence, then \eqref{eq:evp}
is a self-adjoint eigenvalue problem, which has been
studied extensively in the literature, see, e.g., \cite{DFII,Hen06}. 
In particular, it is well known that 
\eqref{eq:evp} has countably many eigenvalues
\[
0 \,<\, \lambda_1 \,<\, \lambda_2 \,\leq\, \lambda_3 \,\leq\, \cdots\,,
\]
and that the smallest eigenvalue is simple. 
However, in our setting the eigenvalues depend on $\bsy$: $\lambda_k = \lambda_k(\bsy)$, 
and since the parameter domain is infinite-dimensional, care
must be taken when transferring classical results to our setting. 
In particular, although it is well known that in the unparametrised setting the
spectral gap, $\lambda_2 - \lambda_1$, is some fixed positive number, 
in our setting the spectral gap, $\lambda_2(\bsy) - \lambda_1(\bsy)$, 
is a function defined on an infinite-dimensional domain,
which could be arbitrarily close to 0.

Our main result (see Theorem~\ref{thm:simple} for a full statement) is as follows.
Assuming that the terms $a_j$ in the coefficient decay sufficiently fast 
(in a suitable norm),
then there exists a $\delta > 0$, independent of $\bsy$, such that
the spectral gap of the eigenvalue problem \eqref{eq:evp} satisfies
\[
\lambda_2(\bsy) - \lambda_1(\bsy) \,\geq\, \delta\,,
\quad \text{for all } \bsy \in U\,.
\]

As an example of the important role that the spectral gap plays in 
error analysis, consider the random elliptic eigenvalue problem from \cite{GGKSS18+}.
There, an algorithm using dimension truncation, Quasi--Monte Carlo (QMC) quadrature
and finite element (FE) methods was used to approximate the
expectation with respect to the stochastic parameters
of the smallest eigenvalue. Throughout the error analysis the 
reciprocal of the spectral gap occurred in: 
1) the bounds on the derivatives of the eigenvalues with respect to $\bsy$ 
(required for the QMC and dimension truncation error analysis); 
2) the constants for the FE error; and 
3) the convergence rate for the eigensolver (by Arnoldi iteration). 
In short, the entire error analysis in \cite{GGKSS18+} fails unless the gap can be bounded from below uniformly in $\bsy$.

In the remainder of this section we frame \eqref{eq:evp} as a variational eigenvalue
problem, introduce the function space setting and summarise some
known properties of the eigenvalues.
Then, in Section~\ref{sec:gap} we prove that the spectral gap is uniformly bounded.
Finally, in Section~\ref{sec:num} we perform a numerical experiment for a specific 
example of \eqref{eq:evp}, and present results on the size of the gap
over different realisations of the parameter generated by a QMC pointset.

\subsection{Variational eigenvalue problems}
\label{sec:var}
It is often useful to study the eigenvalue problem \eqref{eq:evp} in its equivalent variational form.
In this section we introduce the variational eigenvalue problem, then present 
some well known properties and tools that are required for our analysis.

First, we clarify the assumptions on the coefficient $a$ and our setting.

\begin{assumption}
\label{asm:coeff}
\hfill
\begin{enumerate}
\item\label{itm:coeff} The coefficient is of the form
\begin{align}
 a(\bsx, \bsy) \,=\, a_0(\bsx) + \sum_{j = 1}^\infty y_j a_j(\bsx)\,,
\label{eq:a_general}
\end{align}
with $a_j \in L^\infty(D)$, for all $j \ge 0$.
\item\label{itm:amin} There exists $0 < \amin < \amax <\infty$ such that $\amin \leq a(\bsx, \bsy) \leq \amax$,
 for all $\bsx \in D$, $\bsy \in U$.
\item\label{itm:summable} For some $p \in (0, 1)$,
\begin{align*}
\sum_{j = 1}^\infty \nrm{a_j}{L^\infty}^p \,<\, \infty.
\end{align*}
%
\end{enumerate}
\end{assumption}

The last condition (Assumption~\ref{asm:coeff}.\ref{itm:summable})
is the same as is required for the QMC and dimension truncation analysis for
corresponding source problems (see \cite{KSS12}). Note that in that
paper they also allow $p = 1$, but with an extra condition on the size of the sum.

Also, the assumption from the Introduction that each $y_j$ is uniformly
distributed is not a restriction, the important point is that each $y_j$
belongs to a bounded interval.

Let $V = H^1_0(D)$ equipped with the norm $\nrm{v}{V} \coloneqq \nrm{\nabla v}{L^2}$, 
and let $V^*$ denote the dual of $V$.
We identify $L^2(D)$ with its dual, and denote the inner product on $L^2(D)$ 
by $\langle \cdot, \cdot \rangle$, which can be continuously extended 
to a duality pairing on $V \times V^*$, also denoted $\langle \cdot, \cdot \rangle$.
Note that we have the following chain of compact embeddings 
$V \compbed L^2(D) \compbed V^*$.
The parameter domain $U$ is equipped with the topology and metric of $\ell_\infty$.

Next, define the (parametric) symmetric bilinear form $\calA: U \times V \times V \to \R$ by
\begin{align*}
\calA(\bsy; w, v) \,\coloneqq\, \int_D a(\bsx, \bsy) \nabla w(\bsx) \cdot \nabla v(\bsx) \, \rd\bsx\,,
\end{align*}
which is also an inner product on $V$.

In this setting, for each $\bsy \in U$, the variational eigenvalue problem equivalent to \eqref{eq:evp} is:
Find $0 \neq u(\bsy) \in V$ and $\lambda(\bsy) \in \R$ such that
\begin{align}
\label{eq:varevp}
\calA(\bsy; u(\bsy), v) &\,=\, 
\lambda(\bsy) \langle u(\bsy), v \rangle,\quad \text{for all } v \in V\,,\\
\nrm{u(\bsy)}{L^2} &\,=\, 1\,.
\nonumber
\end{align}

It follows from Assumption~A\ref{asm:coeff} that the bilinear form $\calA(\bsy)$ 
is coercive and bounded, uniformly in~$\bsy$:
\begin{align}
\label{eq:A_coerc}
\calA(\bsy; v, v) \,&\geq\, a_{\min}\nrm{v}{V}^2\,, 
\quad &\text{for all } v \in V\,,&
\quad \text{and}\\
\label{eq:A_bounded}
\calA(\bsy; w, v) \,&\leq\, a_{\max}
\nrm{w}{V}\nrm{v}{V}\,,
\quad &\text{for all } w, v \in V\,.
\end{align}

As a consequence, for each $\bsy$ we have a self-adjoint and coercive eigenvalue problem.
Therefore, it is well known (see, e.g., \cite{BO91,Hen06})
that \eqref{eq:varevp} has a countable sequence of positive, real eigenvalues,
which (counting multiplicities) we write as
\[
0 \,<\, \lambda_1(\bsy) \,\leq\, \lambda_2(\bsy) \,\leq\, \cdots\,,
\]
and the corresponding eigenvectors are denoted by $u_1(\bsy), u_2(\bsy), \ldots \in V$.

The \emph{min-max principle} \cite[(8.36)]{BO91}
\[
\lambda_k(\bsy) \,=\, 
\min_{\substack{V_k \subset V\\\dim(V_k) = k}}\max_{v \in V_k} 
\frac{\calA(\bsy; v, v)}{\langle v, v\rangle}\,,
\]
allows us to bound each eigenvalue above and below independently of $\bsy$.
Indeed, by \eqref{eq:A_coerc} and \eqref{eq:A_bounded} we have
\[
\amin\min_{\substack{V_k \subset V\\\dim(V_k) = k}}\max_{v \in V_k} 
\frac{\langle\nabla v, \nabla v \rangle}{\langle v, v \rangle}
\,\leq\, \lambda_k(\bsy) \,\leq\,
\amax \min_{\substack{V_k \subset V\\\dim(V_k) = k}}\max_{v \in V_k} 
\frac{\langle\nabla v, \nabla v \rangle}{\langle v, v \rangle}\,.
\]
Now, using the min-max properties of the $k$th eigenvalue of the Dirichlet Laplacian on $D$, 
which we denote by $\chi_k$, the bounds above simplify to
\begin{equation}
\label{eq:eval_bnd}
\amin \chi_k \,\leq\, \lambda_k(\bsy) \,\leq\, \amax \chi_k\,.
\end{equation}

To consider our problem in the framework of Kato \cite{Kato84} 
for perturbations of linear operators, we introduce, for each $\bsy \in U$,
the solution operator $T(\bsy): V^*\rightarrow V$, which for $f \in V^*$ is 
defined by
\begin{equation}
\label{eq:T_def}
\calA(\bsy; T(\bsy) f, v) \,=\, \langle f, v \rangle
\quad \text{for all } v \in V\,.
\end{equation}
Clearly, $\mu = 1/\lambda$ is an eigenvalue of $T(\bsy)$ if and only if
$\lambda$ is an eigenvalue of \eqref{eq:varevp}, and their eigenspaces coincide.
Alternatively, we can consider the operator $T(\bsy): L^2(D) \to L^2(D)$.
In this case, $T(\bsy)$ is self-adjoint with respect to the $L^2$ inner product 
due to the symmetry of $\calA(\bsy)$; it is compact because $V \compbed L^2(D)$; and
finally, it is bounded due to the Lax-Milgram theorem, which states that for each 
$f \in L^2(D)$ there is a unique $T(\bsy)f \in V$ satisfying \eqref{eq:T_def},
with
\begin{equation}
\label{eq:T_LaxMil}
\nrm{T(\bsy)f}{V} \,\leq\, \frac{1}{\amin} \nrm{f}{V^*} \,\leq\, \frac{1}{\amin\sqrt{\chi_1}}\nrm{f}{L^2}\,.
\end{equation}
In the last inequality, we have used the Poincar\'e inequality: 
\begin{align}\label{eq:poin}
\nrm{v}{L^2} \,\leq\, \evalueLap^{-1/2}\nrm{v}{V} \,,
\quad \text{for } v \in V\,,
\end{align}
and a standard duality argument.
Note that we have expressed the Poincar\'e constant
in terms of the smallest eigenvalue of the Dirichlet Laplacian on $D$,
using again the min-max principle.

\section{Bounding the spectral gap}
\label{sec:gap}
The Krein-Rutman theorem guarantees that for every $\bsy$
the fundamental eigenvalue $\lambda_1(\bsy)$ is simple, see,
e.g., \cite[Theorems 1.2.5 and 1.2.6]{Hen06}. 
However, it does not provide any quantitative
statements about the size of the spectral gap, $\lambda_2(\bsy) - \lambda_1(\bsy)$,
for different parameter values $\bsy$. 
As discussed in the Introduction, when studying numerical methods for eigenvalue problems 
in a UQ setting (see, e.g., \cite{GGKSS18+})
several areas of the error analysis require  {\em uniform positivity} of this  gap over all $\bsy \in U$. 
Here, we prove the required uniform positivity under the conditions of Assumption~A\ref{asm:coeff}, in particular  A\ref{asm:coeff}.\ref{itm:summable}.

An explicit bound on the spectral gap can be obtained
in slightly different settings or 
by assuming tighter restrictions on the coefficients.
For example, for Schr\"odinger operators ($-\Delta + \calV$) on $D$ with a 
weakly convex potential $\calV$ and Dirichlet boundary conditions, 
\cite{AC11} gives an explicit lower bound on the fundamental gap.
Alternatively, using the upper and lower bounds on the eigenvalues \eqref{eq:eval_bnd},
we can determine restrictions on $\amin$ and $\amax$ such that the gap is bounded away from 0.
Explicitly, if the coefficient $a$ is such that $\amin$ and $\amax$ satisfy
\begin{equation}
\label{eq:a_gap_cond}
\frac{\amin}{\amax} \,>\,
\frac{\chi_1}{\chi_2}\,,
\end{equation}
then, by \eqref{eq:eval_bnd}, $\lambda_2(\bsy) - \lambda_1(\bsy) \geq \amin \chi_2 - \amax \chi_1 >  0$.
However, the condition \eqref{eq:a_gap_cond} may prove to be too restrictive.

The general idea of our proof is to use the continuity of the eigenvalues to show
that a non-zero minimum of the gap exists. A complication that arises in this strategy is
that the parameter domain $U$ is not compact, so we cannot immediately 
conclude the existence of such a minumum; we know that $U$ cannot be compact 
in the topology of $\ell_\infty$ because it is the unit ball of $\ell_\infty$,
and the unit ball of an infinite-dimensional Banach space is not compact.
Our solution is based on the fact that although there are infinitely-many parameters,
because of the decay of the terms in the coefficient 
(see Assumption~A\ref{asm:coeff}.\ref{itm:summable}), the contribution of 
a parameter $y_j$ decreases as $j$ increases.
Specifically, we reparametrise \eqref{eq:varevp} as an equivalent eigenvalue problem
whose parameters do belong to a compact set.

The first step is the following elementary lemma, which shows that subsets of   
$\ell^\infty$ that are majorised by an $\ell^q$ sequence (for some  $1 < q< \infty$) are compact.

\begin{Lemma}
\label{lem:compact_subset}
Let $\bsalpha \in \ell^q$ for some  $1 < q < \infty$. The set $U(\bsalpha) \subset \ell^\infty$ 
given by
\begin{equation*}
U(\bsalpha) \,\coloneqq\, \left\{\bsw \in \ell^\infty 
: |w_j| \leq \frac{1}{2} |\alpha_j|\right\}
\end{equation*}
is a compact subset of $\ell^\infty$.
\end{Lemma}

\begin{proof}
Since $\ell^\infty$ is a normed (and hence a metric) space, $U(\bsalpha)$
is compact if and only if it is sequentially compact. To show sequential compactness of $U(\bsalpha)$, 
take any sequence 
$\{\bsy^{(n)}\}_{n \geq 1} \subset U(\bsalpha)$. 
Clearly, by definition of $U(\bsalpha)$,   each $\bsy^{(n)} \in \ell^q$ and moreover, 
\[
\nrm{\bsy^{(n)}}{\ell^q} \,\leq\, \frac{1}{2}\nrm{\bsalpha}{\ell^q} \,<\, \infty
\quad \text{for all } n \in \N\, .
\]
So $\bsy^{(n)}$ is a bounded sequence in $\ell^q$. 
Since $q < \infty$,  $\ell^q$ is a reflexive Banach space, and so  
by \cite[Theorem~3.18]{Brezis11} $\{\bsy^{(n)}\}_{n \geq 1}$ has a subsequence
that converges weakly to a limit in  $\ell^q$. We denote this limit by $\bsy^*$, and,   
with a slight abuse of notation, we denote the convergent subsequence again  by 
$\{\bsy^{(n)}\}_{n \geq 1}$.

We now prove that $\bsy^* \in U(\bsalpha)$ and that the weak convergence is in fact  strong,  
i.e.  we show  $\bsy^{(n)} \rightarrow \bsy^*$ in $\ell^\infty$, as 
$n \rightarrow \infty$.
For any  $j \in \N$, consider the linear functional $f_j:
\ell^q \rightarrow \R$ given by $f_j(\bsw) = w_j$,  where $w_j$
  denotes the $j$th element of the sequence $\bsw = (w_j)_{j \geq 1} \in \ell^q$. 
Clearly, $f_j \in (\ell^q)^*$ (the dual space) and using the weak convergence established above, it follows that  
\[
 y_j^{(n)} \,=\, f_j(\bsy^{(n)}) \,\rightarrow\, f_j(\bsy^*) \,=\, y^*_j 
\quad \text{as } n \rightarrow \infty\,, \quad \text{for each fixed} \ \ j .   
 \] 
That is, we have componentwise convergence. 
 Furthermore, since $|y_j^{(n)}| \leq \frac{1}{2}|\alpha_j|$ it follows that 
 $|y_j^*| \leq \frac{1}{2}|\alpha_j|$ for each $j$,  
 and hence $\bsy^* \in U(\bsalpha)$.
 
 Now, for any $J \in \N$ we can write
\begin{align}
\label{eq:q_norm_split}\nonumber 
\nrm{\bsy^{(n)} - \bsy^*}{\ell^q}^q
\,&=\, 
\sum_{j = 1}^J |y_j^{(n)} - y_j^*|^q
+ \sum_{j = J + 1}^\infty |y_j^{(n)} - y_j^*|^q\\
\,&\leq\,
J\max_{j = 1, 2, \ldots, J} |y_j^{(n)} - y_j^*|^q
+ \sum_{j = J + 1}^\infty |\alpha_j|^q\,.
 \end{align}
 
Let $\varepsilon > 0$. Since $\bsalpha \in \ell^q$, we can choose $J \in \N$ such that
\[
\sum_{j = J + 1}^\infty |{\alpha_j}|^q \,\leq\, \frac{\varepsilon^q}{2}\,,
\]
and since $\bsy^{(n)}$ converges componentwise we can choose $K \in \N$ such that
\[
|y_j^{(n)} - y_j^*| \,\leq\, (2J)^{-1/q} \varepsilon
\quad \text{for all } j = 1, 2, \ldots, J
\text{ and } n \geq K\,.
\]
Thus, by \eqref{eq:q_norm_split} we have $\nrm{\bsy^{(n)} - \bsy^*}{\ell^q}^q \leq \varepsilon^q$
for all $n \geq K$, and hence $\nrm{\bsy^{(n)} - \bsy^*}{\ell^q} \rightarrow 0$ 
as $n \rightarrow \infty$.
Because $\nrm{\bsw}{\ell^\infty} \leq \nrm{\bsw}{\ell^q}$ when 
$\bsw \in \ell^q$  and $1 < q < \infty$, 
this also implies that $\bsy^{(n)} \rightarrow \bsy^*$ in $\ell^\infty$, completing the proof.
\end{proof}

A key property following from the perturbation theory of Kato \cite{Kato84} 
is that the eigenvalues $\lambda_k(\bsy)$ are continuous in $\bsy$,  which 
for completeness is shown 
below in Proposition~\ref{prop:eval_cts}. First, recall that $T(\bsy)$ is 
the solution operator as defined in \eqref{eq:T_def}, and
let $\Sigma(T(\bsy))$ denote the spectrum of $T(\bsy)$.

\begin{Proposition}
\label{prop:eval_cts}
Let Assumption~A\ref{asm:coeff} hold. Then the eigenvalues 
$\lambda_1, \lambda_2 , \ldots$ are Lipschitz continuous in~$\bsy$.
\end{Proposition}

\begin{proof}
We prove the result by establishing the continuity of
the eigenvalues $\mu_k(\bsy)$ of $T(\bsy)$.
Let $\bsy$, $\bsy' \in U$ and consider the operators $T(\bsy), T(\bsy') : L^2(D) \rightarrow L^2(D)$
as defined in \eqref{eq:T_def}.
Since $T(\bsy)$, $T(\bsy')$ are bounded and self-adjoint with respect to $\langle \cdot, \cdot \rangle$,
it follows from \cite[V, \S 4.3 and Theorem~4.10]{Kato84} that we have the
following notion of continuity of $\mu(\cdot)$ in terms of $T(\cdot)$
\begin{equation}
\label{eq:spectrum_cts}
\sup_{\mu \in \Sigma(T(\bsy))} \dist (\mu, \Sigma(T(\bsy')))
\,\leq\, \nrm{T(\bsy) - T(\bsy')}{L^2 \rightarrow L^2}\,.
\end{equation}
%
For an eigenvalue $\mu_k(\bsy) \in \Sigma(T(\bsy))$, \eqref{eq:spectrum_cts} 
implies that there exists a $\mu_{k'}(\bsy') \in \Sigma(T(\bsy'))$ such that
\begin{equation}
\label{eq:spec_cts2}
|\mu_k(\bsy) - \mu_{k'}(\bsy')|
\,\leq\, \nrm{T(\bsy) - T(\bsy')}{L^2 \rightarrow L^2}\,.
\end{equation}
Note that this means there exists an eigenvalue of $T(\bsy')$ close to
$\mu_k(\bsy)$, but does not imply that the $k$th eigenvalue of $T(\bsy')$
is close to $\mu_k(\bsy)$, that is, in \eqref{eq:spec_cts2} $k$ is not necessarily 
equal to $k'$. 
However, consider any $\mu_k(\bsy)$ and let $m$ denote its multiplicity.
Since $m < \infty$, we can assume without loss of generality that the collection 
$\mu_k(\bsy) = \mu_{k + 1}(\bsy) = \cdots = \mu_{k + m - 1}(\bsy)$
is a \emph{finite system} of eigenvalues in the sense of Kato.
It then follows from the discussion in \cite[IV, \S 3.5]{Kato84} that the eigenvalues in this
system depend continuously on the operator with multiplicity preserved.
This preservation of multiplicity is key to our argument, 
since it states that for  $T(\bsy')$ sufficiently close to $T(\bsy)$ there are $m$ 
consecutive eigenvalues 
$\mu_{k'}(\bsy'), \mu_{k' + 1}(\bsy'), \ldots , \mu_{k' + m - 1}(\bsy') \in \Sigma(T(\bsy'))$,
no longer necessarily equal, that are close to $\mu_k(\bsy)$.

A simple argument then shows that
each $\mu_k$ is continuous in the following sense
\begin{equation}
\label{eq:mu_cts_op}
|\mu_k(\bsy) - \mu_k(\bsy')|
\,\leq\, \nrm{T(\bsy) - T(\bsy')}{L^2 \rightarrow L^2}\,.
\end{equation}
To see this, consider, for $k=1,2,\ldots$, the graphs of $\mu_k$ on $U$. Note that the separate graphs can touch (and in principle can even coincide over some subset of $U$), but by definition cannot cross (since at every point in $U$ the successive eigenvalues are nonincreasing); and by the preservation of  multiplicity a graph cannot terminate and a finite set of graphs cannot change multiplicity at an interior point.  Thus by (2.23) the ordered eigenvalues $\mu_k$ must be continuous for each $k \ge 1$  and satisfy \eqref{eq:mu_cts_op}.

It then follows from the relationship $\mu_k(\bsy) = 1/\lambda_k(\bsy)$ along with the upper bound
in \eqref{eq:eval_bnd} that  we have a similar result for the
eigenvalues $\lambda_k$ of \eqref{eq:varevp}:
\begin{equation}
\label{eq:eval_cts_op}
\left|\lambda_k(\bsy) - \lambda_k(\bsy')\right|
\,\leq\, 
(\amax\chi_k)^2\nrm{T(\bsy) - T(\bsy')}{L^2 \rightarrow L^2}\,.
\end{equation}

All that remains is to bound the right hand side of 
\eqref{eq:eval_cts_op} by $\CLip\nrm{\bsy - \bsy'}{\ell^\infty}$,
with $\CLip > 0$ independent of $\bsy$ and $\bsy'$.
To this end, note that since the right hand side
of \eqref{eq:T_def} is independent of $\bsy$ we have
\[
\calA(\bsy; T(\bsy)f, v) \,=\, \calA(\bsy'; T(\bsy')f, v)
\quad 
\text{for all } f \in L^2(D), v \in V\,.
\]
Rearranging and then expanding this gives
\begin{align*}
\calA\left(\bsy; \left(T(\bsy) - T(\bsy')\right)f, v\right))
\,&=\, \calA(\bsy'; T(\bsy')f, v) - \calA(\bsy; T(\bsy')f, v)
\\
\,&=\, \int_D\left[a(\bsx, \bsy') - a(\bsx, \bsy)\right]
\nabla [T(\bsy')f](\bsx) \cdot \nabla v(\bsx)\,\rd\bsx\,.
\end{align*}
Letting $v = (T(\bsy) - T(\bsy'))f \in V$, the left hand side can be bounded from below
using the coercivity \eqref{eq:A_coerc} of $\calA(\bsy)$, 
and the right hand side can be bounded from above using the Cauchy-Schwarz
inequality to give
\[
\amin \nrm{(T(\bsy) - T(\bsy'))f}{V}^2
\,\leq\, \nrm{a(\bsy) - a(\bsy')}{L^\infty}
\nrm{T(\bsy')f}{V}\nrm{(T(\bsy) - T(\bsy'))f}{V}\,.
\]
Dividing by $\amin\nrm{(T(\bsy) - T(\bsy'))f}{V}$ 
and using the upper bound in \eqref{eq:T_LaxMil} we have
\[
\nrm{(T(\bsy) - T(\bsy'))f}{V}
\,\leq\, \frac{1}{\amin^2\sqrt{\chi_1}} \nrm{f}{L^2} \nrm{a(\bsy) - a(\bsy')}{L^\infty}\,.
\]
Then, applying the Poincar\'e inequality \eqref{eq:poin} to the left hand side and
taking the supremum over $f \in L^2(D)$ with $\nrm{f}{L^2} \leq 1$,
in the operator norm we have
\[
\nrm{T(\bsy) - T(\bsy')}{L^2 \rightarrow L^2}
\,\leq\,
\frac{1}{\amin^2\chi_1}\nrm{a(\bsy) - a(\bsy')}{L^\infty}\,.
\]

Using this inequality as an upper bound for \eqref{eq:eval_cts_op} we see
that the eigenvalues inherit the continuity of the coefficient, and so
\begin{equation}
\label{eq:eval_cts_coeff}
|\lambda_k(\bsy) - \lambda_k(\bsy')|
\,\leq\, \frac{\amax^2\chi_k^2}{\amin^2\chi_1}\nrm{a(\bsy) - a(\bsy')}{L^\infty}\,.
\end{equation}
where the constant is clearly independent of $\bsy$ and $\bsy'$.

Finally, to establish Lipschitz continuity with respect to $\bsy$, we 
recall Assumptions~A\ref{asm:coeff}.\ref{itm:coeff} and A\ref{asm:coeff}.\ref{itm:summable},
expand the coefficients in \eqref{eq:eval_cts_coeff} above and use the triangle inequality to give
\[
|\lambda_k(\bsy) - \lambda_k(\bsy')|
\,\leq\, \frac{\amax^2\chi_k^2}{\amin^2\chi_1} \sumj |y_j - y_j'|\nrm{a_j}{L^\infty}
\,\leq\, \frac{\amax^2\chi_k^2}{\amin^2\chi_1}\Bigg(\sumj \nrm{a_j}{L^\infty} \Bigg) 
\nrm{\bsy - \bsy'}{\ell^\infty}\,.
\]
By Assumption~\ref{asm:coeff} the sum is finite, and hence the eigenvalue $\lambda_k(\bsy)$ is Lipschitz in $\bsy$, with the constant clearly independent of $\bsy$.
\end{proof}

Now that we have shown Lipschitz continuity of the eigenvalues and identified suitable compact subsets,
we can prove the main result of this paper:
namely, that the spectral gap is bounded away from 0 uniformly in $\bsy$.
The strategy of the proof is to rewrite 
the coefficient as
\[
a(\bsx, \bsy) \ = \ a_0(\bsx) \ +\  \sumj  \widetilde{y}_j \wa_j(\bsx)\,,
\]
with  $\widetilde{y}_j = \alpha_j y_j$ and $\wa_j(\bsx) = a_j(\bsx)/\alpha_j$, choosing
$\bsalpha \in \ell^q$ to decay slowly enough  
such that $\sumj \Vert \wa_j\Vert_{L^\infty} < \infty$
continues to hold. 
Then, using the intermediate result \eqref{eq:eval_cts_coeff}
from the proof of Proposition~\ref{prop:eval_cts}
we can show that the eigenvalues of the ``reparametrised''  problem  are continuous in the 
new parameter  $\wbsy$,  which now ranges over the compact set $U(\bsalpha)$. 
The required bound on the spectral gap is obtained by using the equivalence of the 
eigenvalues of the original and reparametrised problems. 

\begin{theorem}
\label{thm:simple}
Let Assumption~A\ref{asm:coeff} hold. Then there exists a $\delta > 0$,
independent of $\bsy$, such that
\begin{equation}
\label{eq:gap}
\lambda_2(\bsy) -\lambda_1(\bsy)
\,\geq\, \delta\,.
\end{equation}
\end{theorem}

\begin{proof}
We can assume, without loss of generality, that  $p > 1/2$, 
because if Assumption~A\ref{asm:coeff}.\ref{itm:summable} holds 
with exponent $p' \leq 1/2$ then it also holds for all $p \in (p', 1)$. 
Consequently, set $\eps = 1-p \in (0, 1/2)$ and consider the sequence 
$\bsalpha$ defined by 
\begin{align} \label{eq:defalpha} 
\alpha_j \,=\, \Vert a_j\Vert ^\eps_{L^\infty} + 1/j, \quad 
\text{for each} \ \ j \in \N .   
\end{align}
Setting  $q = p/\eps = p/(1-p) \in (1,\infty)$, using Assumption A\ref{asm:coeff}.3 and  
the triangle inequality, it is easy to see that 
$\bsalpha \in \ell^q$. Moreover, the inclusion of $1/j$ in
\eqref{eq:defalpha} ensures that  $\alpha_j \not= 0$, for all $j
\geq 1$. Hence, from now on, for  $\bsw = (w_j)_{j = 1}^\infty \in \ell^\infty$, 
 we can define the sequences  $\bsalpha \bsw = (\alpha_j w_j)_{j=1}^\infty$ and  $\bsw/\bsalpha = (w_j/\alpha_j)_{j=1}^\infty$. 
Then, recalling the definition of $U(\bsalpha)$ in Lemma \ref{lem:compact_subset}, 
it is easy to see that$\wbsy \in U(\bsalpha)$ if and only if $\wbsy/\bsalpha \in U$,
and moreover, $\bsy \in U$ if and only if $\bsalpha \bsy \in U(\bsalpha)$.

Now for $\bsx \in D$ and $\wbsy \in U(\bsalpha)$, we define
\[
\wa(\bsx, \wbsy) = a_0(\bsx) + \sum_{j=1}^\infty \widetilde{y}_j \frac{a_j(\bsx)}{\alpha_j}\,,
\]
from which it is easily seen that 
\begin{equation}
\label{eq:coeff_1to1}
\wa(\bsx, \wbsy) = a (\bsx, \wbsy/\bsalpha)\,.
\end{equation}
Then we set 
\[
\widetilde{\calA}(\widetilde{\bsy}; w, v)
\,\coloneqq\, \int_D \widetilde{a}(\bsx, \widetilde{\bsy}) \nabla w(\bsx) . \nabla v(\bsx) \, \rd \bsx
\quad \text{for } w, v \in V\, ,
\]
and we consider the following reparametrised eigenvalue problem: 
Find $\widetilde{\lambda}(\widetilde{\bsy}) \in \R$
and $0 \neq \widetilde{u}(\widetilde{\bsy}) \in V$ such that
\begin{align}
\label{eq:reparam_evp}\nonumber
\widetilde{\calA}(\widetilde{\bsy}; \widetilde{u}(\widetilde{\bsy}), v)
\,&=\, \widetilde{\lambda}_k(\widetilde{\bsy}) \langle\widetilde{u}(\widetilde{\bsy}), v\rangle
\quad \text{for all } v \in V\,,\\
\nrm{\widetilde{u}(\widetilde{\bsy})}{L^2} \,&=\, 1\,.
\end{align}
Note that because we have equality between the original and reparametrised coefficients 
\eqref{eq:coeff_1to1}, for each $\bsy \in U$, and corresponding $\wbsy = \bsalpha \bsy\in U(\bsalpha)$,
\eqref{eq:coeff_1to1} implies
that there is equality between eigenvalues
$\lambda_k(\bsy)$ of \eqref{eq:varevp} and  
$\wlam_k(\wbsy)$ of 
the reparametrised eigenvalue problem \eqref{eq:reparam_evp}
\begin{equation}
\label{eq:eval_1to1}
\lambda_k(\bsy) \,=\, \widetilde{\lambda}_k(\wbsy)
\quad \text{for } k \in \N\,,
\end{equation}
and their eigenspaces coincide.

Moreover, for an eigenvalue $\wlam_k(\wbsy)$ of \eqref{eq:reparam_evp}, 
using \eqref{eq:eval_1to1} in the inequality \eqref{eq:eval_cts_coeff} we have
\[
|\widetilde{\lambda}_k(\widetilde{\bsy}) - \widetilde{\lambda}_k(\widetilde{\bsy}')|
\,\leq\, \frac{\amax^2\chi_k^2}{\amin^2\chi_1} 
\nrm{a(\wbsy/\bsalpha) - a(\wbsy/\bsalpha)}{L^\infty}\,,
\]
which after expanding the coefficient and using the triangle inequality becomes
\[
|\widetilde{\lambda}_k(\widetilde{\bsy}) - \widetilde{\lambda}_k(\widetilde{\bsy}')|
\,\leq\,
\Bigg(\underbrace{\frac{\amax^2\chi_k^2}{\amin^2\chi_1} \sum_{j = 1}^\infty 
\frac{1}{\alpha_j}\nrm{a_j}{L^\infty}}_{\wCLip}\Bigg)
\nrm{\wbsy - \wbsy'}{\ell^\infty}\,,
\]
where $\wCLip$ is clearly independent of $\wbsy$ and $\wbsy'$.
Now by \eqref{eq:defalpha} together with Assumption~A\ref{asm:coeff}, 
we have 
\[
\sumj \frac{\Vert a_j\Vert_{L^\infty}}{\alpha_j} \ \leq \ \sumj  \Vert a_j\Vert_{L^\infty}^{1-\eps} \ = \  \sumj  \Vert a_j\Vert_{L^\infty}^{p} <\  \infty \,.
\] 
Thus, $\wCLip < \infty$ and hence the reparametrised eigenvalues 
are continuous on $U(\bsalpha)$.

It immediately follows that the spectral gap 
$\widetilde{\lambda}_2(\widetilde{\bsy}) - \widetilde{\lambda}_1(\widetilde{\bsy})$
is also continuous on $U(\bsalpha)$, which by Lemma~\ref{lem:compact_subset} is a 
compact subset of $\ell^\infty$. Therefore, the non-zero
minimum is attained giving that the spectral gap  
$\widetilde{\lambda}_2(\widetilde{\bsy}) - \widetilde{\lambda}_1(\widetilde{\bsy})$ 
is uniformly positive.
Finally, because there is equality between the original 
and reparametrised eigenvalues \eqref{eq:eval_1to1}
the result holds for the original problem over all $\bsy \in U$.
\end{proof}

The paper \cite{GGKSS18+} proves a similar result to Theorem~\ref{thm:simple} in
a more general setting, the same strategy is used there also.

\section{Numerical results}
\label{sec:num}
For our numerical results, we consider the 1-dimensional domain $D = (0, 1)$, and 
let the basis functions that define the coefficient in \eqref{eq:a_general} be given by
\begin{equation}
\label{eq:a_j}
 a_j(x) \,=\, 
\begin{cases}
a_0 &\text{for } j = 0,\\[2mm]
\displaystyle \frac{c_0}{j^2} \sin(j\pi x)\,,  &\text{for } j = 1, 2, \ldots, s,\\[2mm]
0 & \text{for } j > s,
\end{cases}
\end{equation}%
where $a_0 > 0$ and $c_0\in \R$ will determine the values of
$\amin$ and $\amax$.  Note that the stochastic dimension is
here fixed at $s = 100$.
Since multiplying the coefficient by any constant factor will simply
rescale the eigenvalues by that same factor, without loss of generality
we can henceforth set $a_0 = 1$ and vary $c_0$.

For a coefficient \eqref{eq:a_general} given by the basis functions \eqref{eq:a_j},
we can obtain a formula for the bounds $\amin$, $\amax$ as follows
\begin{align*}
\amin \,=\, 1 - \frac{c_0}{2} \sum_{j = 1}^s \frac{1}{j^2} \,=\, 1 - 0.81c_0\,,\\
\amax \,=\, 1 + \frac{c_0}{2} \sum_{j = 1}^s \frac{1}{j^2} \,=\, 1 + 0.81c_0\,.
\end{align*}
We remark that these bounds are not sharp.


We also consider the so-called \emph{log-normal} coefficient:
\begin{equation}
\label{eq:lognormal}
a(x, \bsy) = a_* + \exp\Bigg(\sum_{j = 1}^{s} \Phi^{-1}(y_j + \tfrac{1}{2}) a_j(x)\Bigg)\,,
\end{equation}
where $a_* \geq 0$, $a_j$ are as in \eqref{eq:a_j}, and 
$\Phi^{-1}$ is the inverse of the normal cumulative distribution function. 
In this way each $\Phi^{-1}(y_j + \tfrac{1}{2}) \sim\calN(0, 1)$. 
Since $\Phi^{-1}$ maps $[0, 1]$ to $\R$, the coefficient \eqref{eq:lognormal}
is unbounded, and although it is positive ($\amin = a_*$), 
for $a_* = 0$ it could be arbitrarily close to 0.
So in this case Assumption~\ref{asm:coeff}.\ref{itm:amin} does not hold, 
our compactness argument fails and we have no theoretical prediction.

To approximate the eigenvalue problem in space we use piecewise linear finite elements
on a uniform mesh, with a meshwidth of $h = 1/64$. Numerical tests for different
meshwidths ($h = 1/8$ to $h = 1/128$) produced qualitatively the same results.

The purpose of this section is to provide supporting evidence
that for the problem above the gap remains bounded.  
We do this in a brute force manner by studying the minimum of the gap
over a large number of parameter realisations.
The parameter realisations are generated by a QMC point set; specifically, a base-2 embedded
lattice rule (see \cite{CKN06}) with up to $2^{20}$ points and a single random shift. 
To generate the points we use the generating vector \texttt{exod2\_base2\_m20\_CKN} from \cite{DirkMPS}.
We choose QMC points as the test set
because they can be shown to be well-distributed in high dimensions, see e.g. \cite{CKN06,DKS13}.
The goal is not to find the minimum, but to provide evidence that as more and more of the
parameter domain is searched (which corresponds to more realisations), the minimum of the gap 
over all realisations approaches a constant value.

The results are given in Figures~\ref{fig:uni1}--\ref{fig:LN2}, where in each figure we plot the
minimum of the gap against the number of realisations $N$ (blue circles) and an 
estimate (see the next paragraph) of the distance between the estimated minimum and the true minimum (black triangles), 
along with a least-squares fit to $\alpha N^{-\beta}$ (dashed red line) of this estimate. 
Each data point corresponds to a doubling of the number of QMC points:
$N = 1, 2, 4, \ldots, 2^{20}$, and the axes are in loglog scale.

Letting $\delta_N$ denote the approximate minimum over the first $N$ realisations,
we estimate the distance to the true minimum by 
\[
\delta_N - \min_{\bsy \in U} \big(\lambda_2(\bsy) - \lambda_1(\bsy) \big)
\,\approx\, \delta_N - \delta_{N^*}\,,
\]
where $\delta_{N^*}$ corresponds to the most accurate estimate of the minimum, 
with $N^* = 2^{20}$. The purpose of including such an
estimate of the distance to the true minimum is to demonstrate
that not only does the minimum of the gap appear to plateau, but
that the differences also decay like a power of $N$.

First we consider the affine coefficient in \eqref{eq:a_general}.
The three different choices of $c_0$ are: 
in Figure~\ref{fig:uni1}  $c_0 = 1$, 
which gives $\amin = 0.18$ and $\amax = 1.82$;
in Figure~\ref{fig:uni2} $ c_0 = 1.223$, 
which gives $\amin = 2\times10^{-4}$ and $\amax = 2$; and
in Figure~\ref{fig:uni3}   $c_0 = 0.5$, 
which gives $\amin = 0.59$ and $\amax = 1.41$.
Then, in Figure~\ref{fig:LN1}  we plot the log-normal 
coefficient \eqref{eq:lognormal} with $\amin = a_* = 0$ and $c_0 = 1$, and
in Figure~\ref{fig:LN2} we plot the lognormal coefficient with
$\amin = a_* = 0.18$ and $c_0 = 1$.

For each different choice of the affine coefficient \eqref{eq:a_general} 
(Figures~\ref{fig:uni1}, \ref{fig:uni2}, \ref{fig:uni3}),
the minimum value of the gap appears to plateau and approach a nonzero minimum. 
The minimum of the gap seems fairly insensitive to changes in $c_0$.
However, in Figures~\ref{fig:LN1} and  \ref{fig:LN2} for the log-normal coefficient 
(which, we recall is not covered by the theory of the current work)
the results are inconclusive.
It appears that the gap tends to 0 in the case of a true lognormal
coefficient ($a_* = 0$, Figure~\ref{fig:LN1}), and that it is bounded away from 0
for the ``regularised'' lognormal coefficient with $a_* = 0.18$ in Figure~\ref{fig:LN2}.
However, in Figure~\ref{fig:LN1} the smallest computed value of the gap is close to 1, 
and it is possible that it will plateau for a denser set of QMC points.
Also, in Figure~\ref{fig:LN2} the plateau is not as clearly developed in 
as in Figures~\ref{fig:uni1}--\ref{fig:uni3}.
It remains an open question if the gap can be bounded in the
case of the lognormal coefficient \eqref{eq:lognormal}, and whether
$a_*$ needs to be strictly positive.

\begin{figure}[!h]
\centering
\includegraphics[scale=.36]{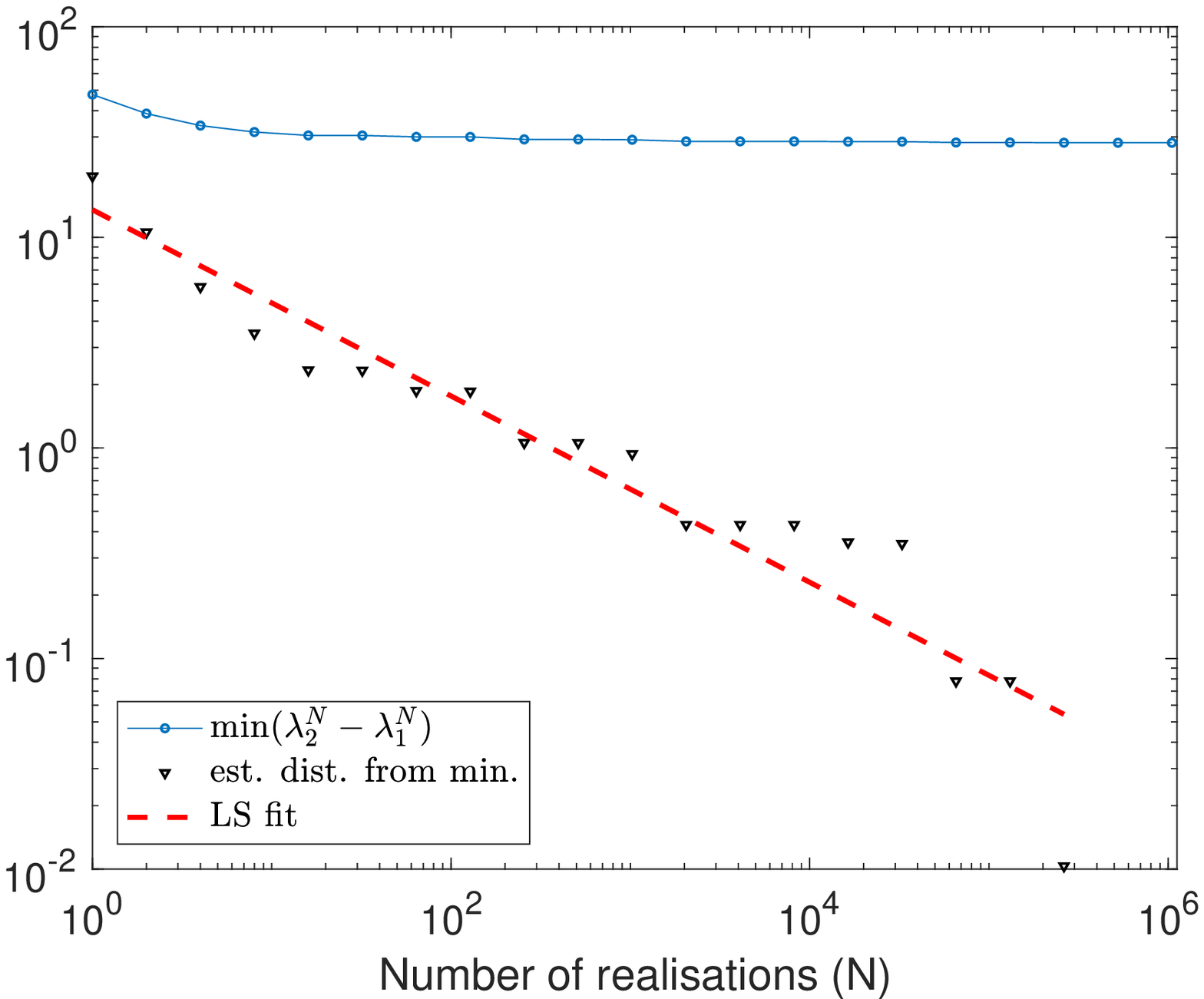}
\caption{Estimate of the minimum of the spectral gap and estimate of the distance to the true minimum: affine coefficient \eqref{eq:a_general} with $a_0 = c_0 = 1$, $\amin = 0.18$, $\amax = 1.82$.}
\label{fig:uni1}
\end{figure}

\begin{figure}[!h]
\centering
\includegraphics[scale=.36]{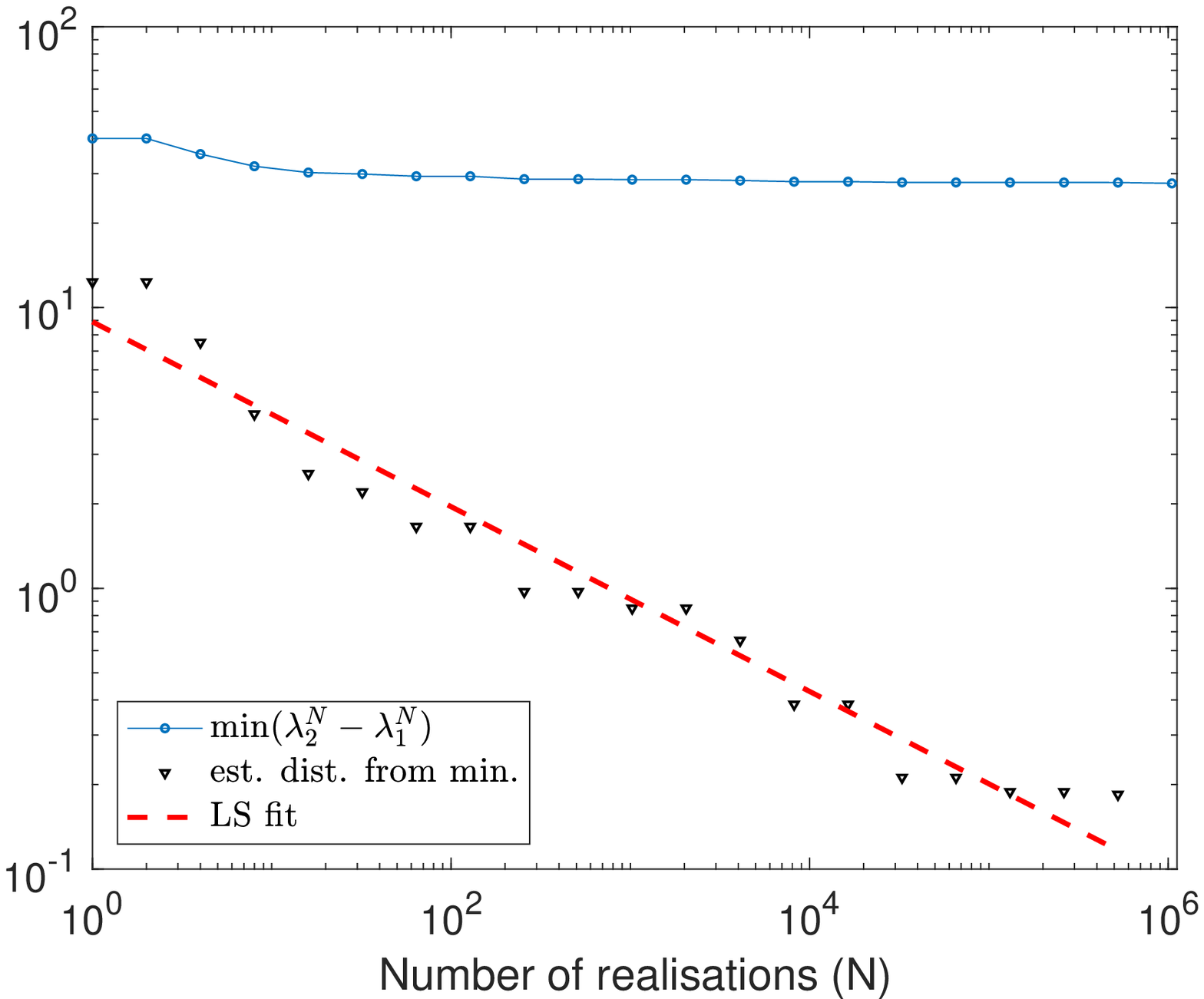}
\caption{Estimate of the minimum of the spectral gap and estimate of the distance to the true minimum: affine coefficient \eqref{eq:a_general} with $a_0 = 1$, $c_0 = 1.223$, $\amin = 2\times10^{-4}$, $\amax = 2$.}
\label{fig:uni2}
\end{figure}

\begin{figure}
\centering
\includegraphics[scale=.36]{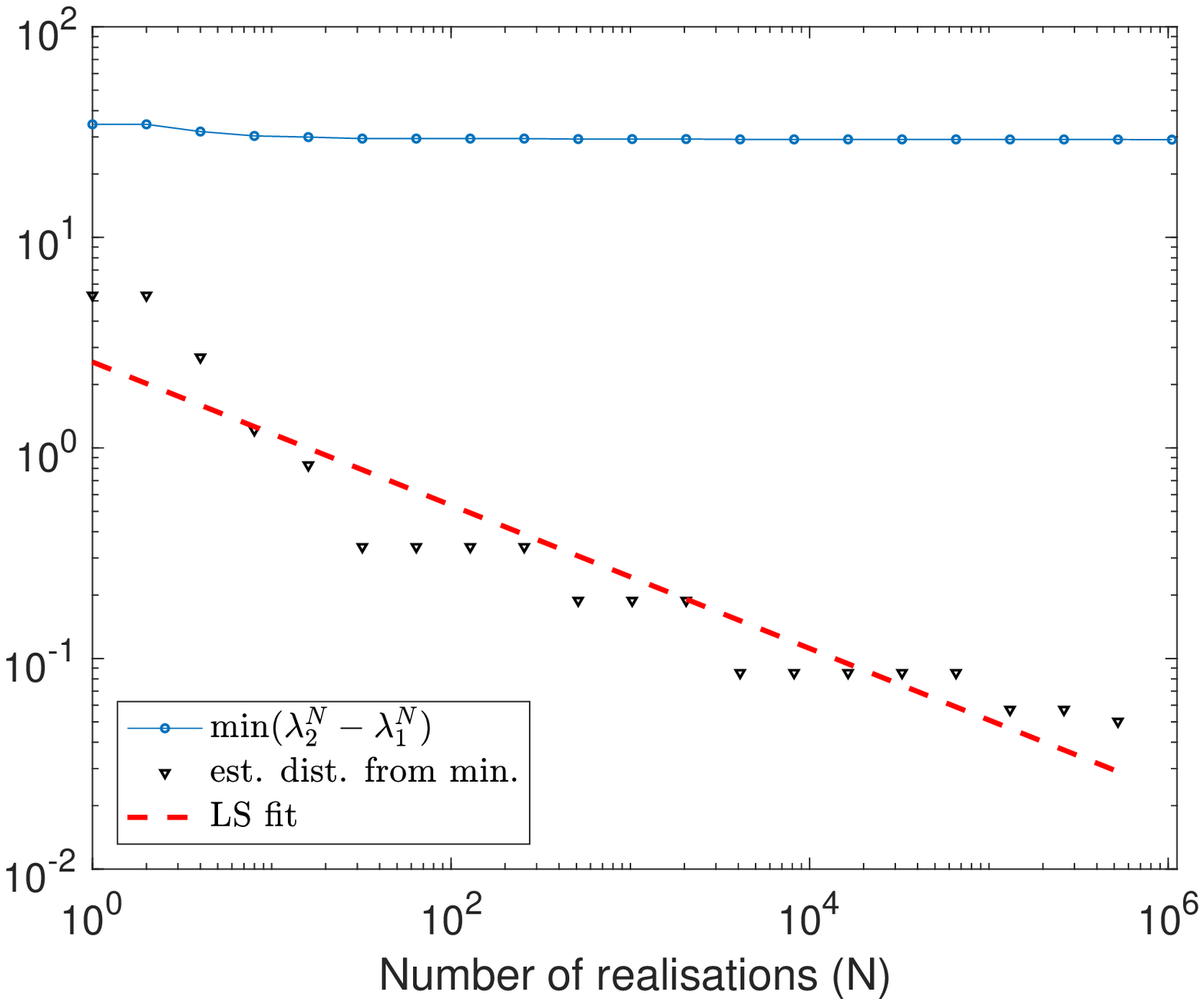}
\caption{Estimate of the minimum of the spectral gap and estimate of the distance to the true minimum: affine coefficient \eqref{eq:a_general} with $a_0 = 1$, $c_0 = 0.5$, $\amin = 0.59$, $\amax = 1.41$.}
\label{fig:uni3}
\end{figure}

\begin{figure}
\centering
\includegraphics[scale=.36,trim={1mm 0 1mm 0.5mm}]{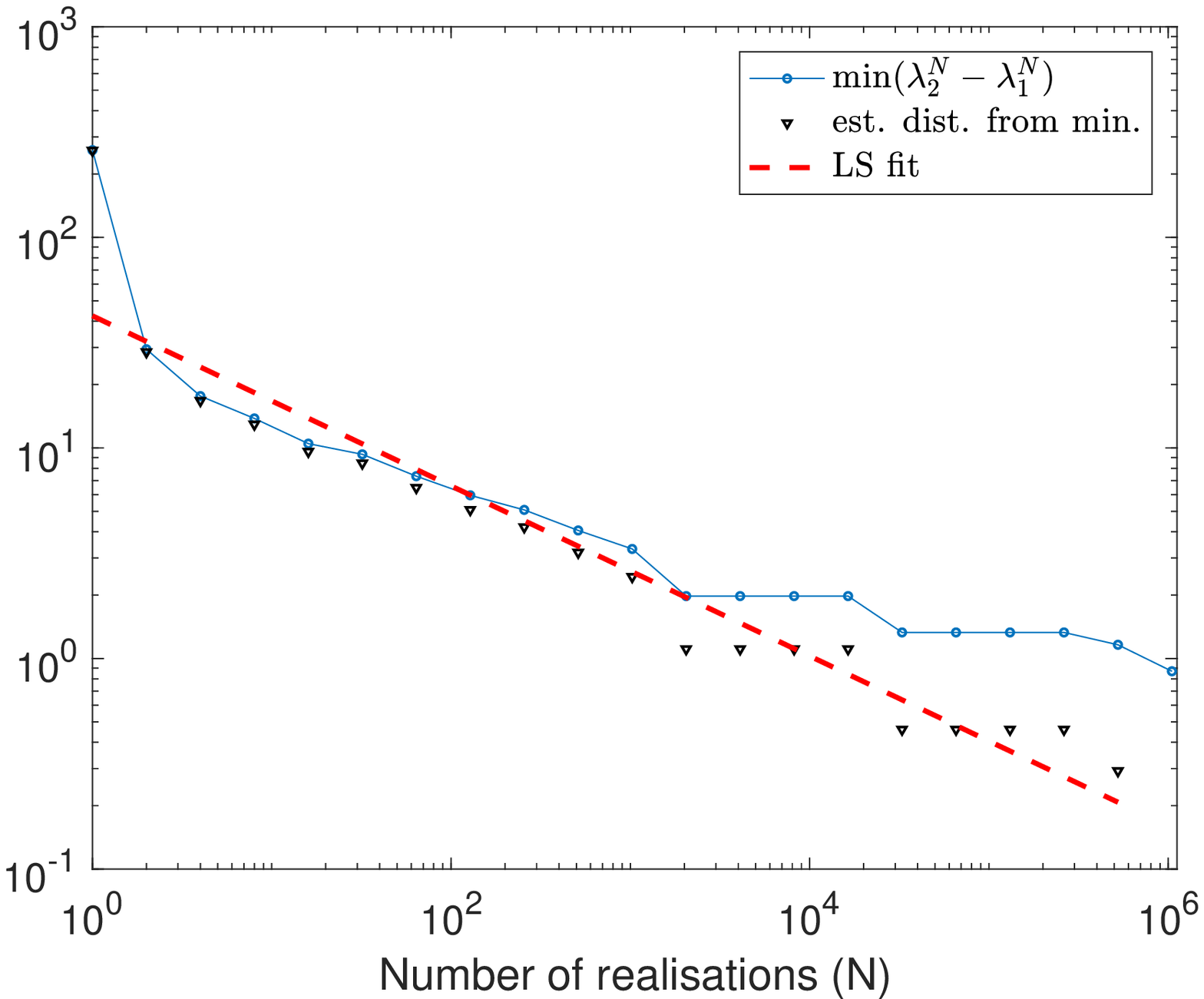}
\caption{Estimate of the minimum of the spectral gap and estimate of the distance to the true minimum: log-normal coefficient \eqref{eq:lognormal} with $\amin = a_* = 0$ and $c_0 = 1$.}
\label{fig:LN1}
\end{figure}

\begin{figure}
\centering
\includegraphics[scale=.36,trim={1mm 0 1mm 0.5mm}]{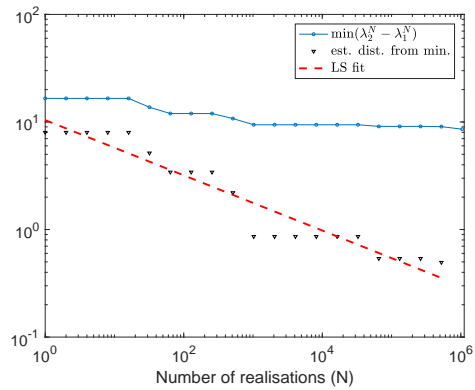}
\caption{Estimate of the minimum of the spectral gap and estimate of the distance to the true minimum: log-normal coefficient \eqref{eq:lognormal} with $\amin = a_* = 0.18$ and $c_0 = 1$.}
\label{fig:LN2}
\end{figure}

\pagebreak
\section{Conclusion}
The spectral gap is an important quantity that occurs throughout several areas
of the numerical analysis of eigenvalue problems, and in this work we proved that,
under certain conditions on the coefficient,
the spectral gap of a random elliptic eigenvalue problem is 
uniformly bounded from below.
In all of our numerical experiments the results strongly suggest that the 
minimum of the gap approaches a nonzero constant value.
The only exception is the lognormal coefficient, which is not covered by our theory.

\bibliographystyle{plain}

\end{document}